\documentclass{ws-jktr}
\usepackage{hyperref}
\usepackage{float}
\floatstyle{plaintop}
\restylefloat{table}

\newcommand{\C}{\mathbb{C}}

\newcommand{\Z}{\mathbb{Z}}

\newcommand{\Hy}{\mathbb{H}}

\begin{document}

\markboth{Ji-Young Ham, Joongul Lee}
{Volumes of $C(2n,3)$}

\catchline{}{}{}{}{}

\title{The volume of hyperbolic cone-manifolds of the knot with Conway's notation $C(2n, 3)$
}

\author{Ji-Young Ham 
\footnote{This work was supported by Basic Science Research Program through the National Research Foundation of Korea (NRF) funded by the Ministry of Education, Science and Technology (No. NRF-2008-341-C00004 and NRF-R01-2008-000-10052-0).}
}

\address{Department of Science, Hongik University, 
94 Wausan-ro, Mapo-gu, Seoul,
 04066,\\
 Department of Mathematical Sciences, Seoul National University, 
 1 Gwanak-ro, Gwanak-gu, Seoul 08826,\\
  Korea\\
   jiyoungham1@gmail.com}

\author{Joongul Lee
\footnote{This work was supported by 2015 Hongik University Research Fund.}
}

\address{Department of Mathematics Education, Hongik University, 
94 Wausan-ro, Mapo-gu, Seoul,
04066,\\
   Korea\\
   jglee@hongik.ac.kr}

\maketitle

\begin{abstract}
Let $C(2n, 3)$ be the family of two bridge knots of slope $(4n+1)/(6n+1)$. We calculate the volumes of the $C(2n, 3)$ cone-manifolds using the Schl\"{a}fli formula. We present the concrete and explicit formula of them. We apply the general instructions of Hilden, Lozano, and Montesinos-Amilibia and extend the Ham, Mednykh, and Petrov's methods. As an application, we give the volumes of the cyclic coverings over those knots. For the fundamental group of $C(2n, 3)$, we take and tailor Hoste and Shanahan's. As a byproduct, we give an affirmative answer for their question whether their presentation is actually derived from Schubert's canonical 2-bridge diagram or not.
\end{abstract}



\section{Introduction}

Let us denote $C(2n,3)$ by $T_{2n}$ and  the exterior of $T_{2n}$ by $X_{2 n}$. By Mostow-Prasad rigidity, $X_{2n}$ has a unique hyperbolic structure. 
Let $\rho_{\infty}$ be the holonomy representation from $\pi_1(X_{2n})$ to $PSL(2,C)$ and denote $\rho_{\infty} \left(\pi_1(X_{2n})\right)$ by $\Gamma$, a Kleinian group.
$X_{2n}$ is a $(PSL(2,C),\Hy^3)$-manifold and can be identified with 
$\Hy^3/ \Gamma$. 
Thurston's orbifold theorem guarantees an orbifold, $X_{2n}(\alpha)$, with $T_{2n}$ as the singular locus and the cone-angle 
$\alpha=2 \pi/k$ for some nonzero integer $k$ can be identified with 
$\Hy^3/\Gamma^{\prime}$ for some $\Gamma^{\prime} \in PSL(2,C)$; the hyperbolic structure of $X_{2n}$ is deformed to the hyperbolic structure of 
$X_{2n}(\alpha)$.
For the intermediate angles whose multiples are not $2 \pi$ and not bigger than $\pi$, Kojima~\cite{K1} showed that the hyperbolic structure of 
$X_{2n}(\alpha)$ can be obtained uniquely by deforming nearby orbifold structures.
Note that there exists an angle $\alpha_0 \in [\frac{2\pi}{3},\pi)$ for each $T_{2n}$ such that $X_{2n}(\alpha)$ is hyperbolic for $\alpha \in (0, \alpha_0)$, Euclidean for $\alpha=\alpha_0$, and spherical for $\alpha \in (\alpha_0, \pi]$ \cite{P2,HLM1,K1,PW}. 
For further knowledge of cone manifolds you can consult~\cite{CHK,HMP}.

Explicit volume formulae for hyperbolic cone-manifolds of knots and links are known a little. The volume formulae for hyperbolic cone-manifolds of the knot 
$4_1$~\cite{HLM1,K1,K2,MR1}, the knot $5_2$~\cite{M2}, the link $5_1^2$~\cite{MV1}, 
the link $6_2^2$~\cite{M1}, and the link $6_3^2$~\cite{DMM1} have been computed. In~\cite{HLM2} a method of calculating the volumes of two-bridge knot cone-manifolds were introduced but without explicit formulae. In~\cite{HMP}, explicit volume formulae for hyperbolic twist knot cone-manifolds are computed. Similar methods are used for computing Chern-Simons invariants of twist knot orbifolds in~\cite{HL}.

The main purpose of the paper is to find an explicit and efficient volume formula for hyperbolic cone-manifolds of the knot $T_{2n}$. The following theorem gives the volume formula for 
$X_{2n}(\alpha)$.
Note that if $2 n$ of $T_{2n}$ is replaced by an odd integer, then $T_{2n}$ becomes a link with two components. Also, note that the volume of hyperbolic cone-manifolds of the knot with Conway's notation $C(-2n, -3)$ is the same as that of the knot with Conway's notation $C(2n,3)$. For the volume formula, since the knot $T_{2n}$ has to be hyperbolic, we exclude the case when $n=0$.

\begin{theorem}\label{thm:main}
Let $X_{2n}(\alpha)$, $0 \leq \alpha < \alpha_0$ be the hyperbolic cone-manifold with underlying space $S^3$ and with singular set $T_{2n}$ of cone-angle 
$\alpha$. Then the volume of $X_{2n}(\alpha)$ is given by the following formula

\begin{align*}
\text{\textnormal{Vol}} \left(X_{2n}(\alpha)\right) &= \int_{\alpha}^{\pi} \log \left|\frac{M^{-2}+x}{ M^{2}+x}\right| \: d\alpha,
\end{align*}

\noindent where 
$x$ with $\text{\textnormal{Im}}(x) \leq 0$ is a zero of the Riley-Mednykh polynomial $P_{2n}=P_{2n}(x,M)$ which is given recursively by 

\medskip
\begin{equation*}
P_{2n} = \begin{cases}
 Q P_{2(n-1)} -M^8 P_{2(n-2)} \ \text{if $n>1$}, \\
 Q P_{2(n+1)}-M^8 P_{2(n+2)} \ \text{if $n<-1$},
\end{cases}
\end{equation*}
\medskip

\noindent with initial conditions
\begin{equation*}
\begin{split}
P_{-2}  & =M^2 x^2+\left(M^4- M^2+1\right) x+M^2,\\
P_{0}  & =M^{-2} \ \text{\textnormal{for}} \ n < 0 \qquad \text{\textnormal{and}} \qquad P_{0} (x,M)  =1 \ \text{\textnormal{for}} \ n > 0,\\    
P_{2}  & =-M^4 x^3+\left(-2M^6+M^4-2 M^2\right) x^2+\left(-M^8+M^6-2 M^4+ M^2-1\right) x+M^4, \\
\end{split}
\end{equation*}
\noindent and $M=e^{\frac{i \alpha}{2}}$ and
$$Q=-M^4 x^3+\left(-2M^6+2 M^4-2M^2\right) x^2+\left(-M^8+2M^6-3M^4+2M^2-1\right) x+2M^4.$$
\end{theorem}

From Theorem~\ref{thm:main}, the following corollary can be obtained. The following corollary gives the hyperbolic volume of the $k$-fold cyclic covering over $T_{2n}$, $M_k (X_{2n})$, for $k \geq 3$.

\begin{corollary}
The volume of $M_k (X_{2n})$ is given by the following formula

\begin{align*}
\text{\textnormal{Vol}} \left(M_k (X_{2n}))\right) &=k \int_{\frac{2 \pi}{k}}^{\pi} \log \left|\frac{M^{-2}+x}{ M^{2}+x}\right| \: d\alpha,
\end{align*}

\noindent where 
$x$ with $\text{\textnormal{Im}}(x) \leq 0$ is a zero of the Riley-Mednykh polynomial $P_{2n}=P_{2n}(x,M)$ and $M=e^{i \frac{\alpha}{2}}$. 
\end{corollary}

\section{Two bridge knots with  Conway's notation $C(2n, 3)$} \label{sec:C[2n,3]}

\begin{figure} 
\begin{center}
\resizebox{5cm}{!}{\includegraphics{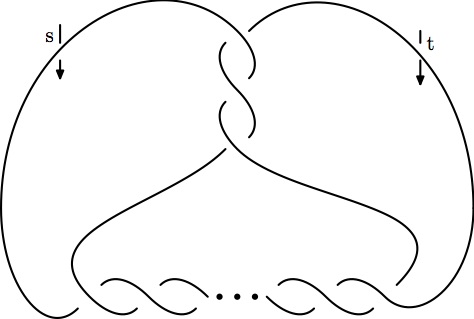}}
\qquad
\resizebox{5cm}{!}{\includegraphics{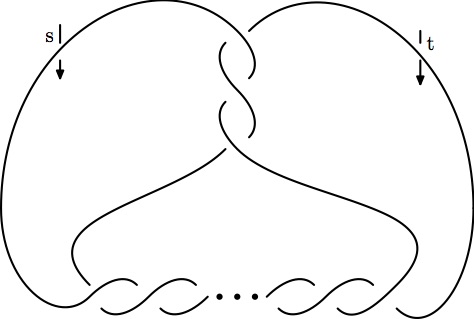}}
\caption{A two bridge knot with Conway's notation $C(2n,3)$  for $n>0$ (left) and for $n<0$ (right)} \label{fig:C[2n,3]}
\end{center}
\end{figure}

\begin{figure}
\begin{center}
\resizebox{5cm}{!}{\includegraphics{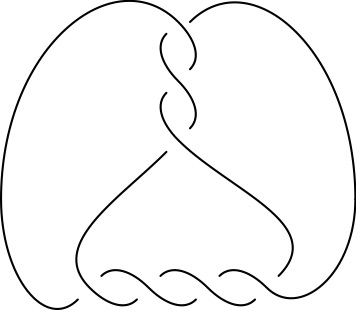}}
\caption{The knot $C(4,3)$ ($7_3$ in the Rolfsen's knot table)}\label{fig:knot}
\end{center}
\end{figure}

A knot is a two bridge knot with Conway's notation $C(2n,3)$ if it has a regular two-dimensional projection of the form in Figure~\ref{fig:C[2n,3]}. For example, Figure~\ref{fig:knot} is knot $C(4,3)$.
It has 3 left-handed vertical crossings and $2 n$ right-handed horizontal crossings. 
Recall that we denote it by $T_{2 n}$. 
In~\cite[Proposition 1]{R1}, the fundamental group of two-bridge knots is presented. We will use the fundamental group of $X_{2n}$ in~\cite{HS}. In~\cite{HS}, the fundamental group of $X_{2n}$ is calculated with 3 right-handed vertical crossings as positive crossings instead of three left-handed vertical crossings. The following proposition is tailored to our purpose. The reduced word relation of the one in the following proposition can also be obtained by reading off the fundamental group from the Schubert normal form of $T_{2n}$ with slope $\frac{4n+1}{6n+1}$ for $n >0$ and by taking the inverse of both sides after reading off the fundamental group from the Schubert normal form of $T_{2n}$ with slope $\frac{4n+1}{6n+1}$ for $n<0$~\cite{S,R1}, which answers the Hoste-Shanahan's question whether their presentation is actually derived from Schubert's canonical 2-bridge diagram or not.

\begin{proposition}\label{thm:fundamentalGroup}
$$\pi_1(X_{2n})=\left\langle s,t \ |\ swt^{-1}w^{-1}=1\right\rangle,$$
where $w=(ts^{-1}tst^{-1}s)^n$.
\end{proposition}
\section{$\left(PSL(2,C),\mathbb{H}^3 \right) \text{ structure of } X_{2n} (\alpha)$} \label{sec:poly}
Let $R=\text{Hom}(\pi_1 (X_{2n}), \text{SL}(2, \C))$. 
Given  a set of generators, $s,t$, of the fundamental group for 
$\pi_1 (X_{2n})$, we define
 a set $R\left(\pi_1 (X_{2n})\right) \subset \text{SL}(2, \C)^2 \subset \C^{8}$ to be the set of
 all points $(\rho(s),\rho(t))$, where $\rho$ is a
 representaion of $\pi_1 (X_{2n})$ into $\text{SL}(2, \C)$. Since the defining relation of 
 $\pi_1 (X_{2n})$ gives the defining equation of $R\left(\pi_1 (X_{2n})\right)$~\cite{R3}, $R\left(\pi_1 (X_{2n})\right)$ is an affine algebraic set in $\C^{8}$. 
$R\left(\pi_1 (X_{2n})\right)$ is well-defined up to isomorphisms which arise from changing the set of generators. We say elements in $R$ which differ by conjugations in $\text{SL}(2, \C)$ are \emph{equivalent}. 
A point on the variety gives the $\left(PSL(2,C),\mathbb{H}^3 \right) \text{ structure of } X_{2n} (\alpha)$.

Let
\begin{center}
$$\begin{array}{ccccc}
\rho(s)=\left[\begin{array}{cc}
                       M &       1 \\
                        0      & M^{-1}  
                     \end{array} \right]                          
\text{,} \ \ \
\rho(t)=\left[\begin{array}{cc}
                   M &  0      \\
                   2-M^2-M^{-2}-x      & M^{-1} 
                 \end{array}  \right].
\end{array}$$
\end{center}

Then $\rho$ becomes a representation if and only if $M$ and $x$ satisfies a polynomial equation~\cite{R3,MR2}. We call the defining polynomial 
of the algebraic set $\{(M,x)\}$ as the \emph{Riley-Mednykh polynomial}.

\subsection{The Riley-Mednykh polynomial}
Since we are interested in the excellent component (the geometric component) of $R\left(\pi_1(X_{2n})\right)$, 
in this subsection we set 
$M=e^{\frac{i \alpha}{2}}$.
Given the fundamental group of $X_{2n}$,
$$\pi_1(X_{2n})=\left \langle s,t \ |\  swt^{-1}w^{-1}=1 \right \rangle,$$
where $w=(ts^{-1}tst^{-1}s)^n$, let $S=\rho(s),\  T=\rho(t)$ and $W=\rho(w)$. Then the trace of $S$ and the trace of $T$ are both 
$2 \cos \frac{\alpha}{2}$. 

\begin{lemma}\label{lem:swc}
For $c \in \text{\textnormal{SL}}(2, \C)$ which satisfies $cS=T^{-1}c$ and $c^2=-I$,
$$SWT^{-1}W^{-1}=-(SWc)^2 $$
\end{lemma}
\begin{proof}
 \begin{equation*}
\begin{split}
 (SWc)^2 & =SWcSWc=SWT^{-1}c(TS^{-1}TST^{-1}S)^nc \\
              & =SWT^{-1}(S^{-1}TS^{-1}T^{-1}ST^{-1})^nc^2=-SWT^{-1}W^{-1}.
 \end{split}
\end{equation*}   
\end{proof}

From the structure of the algebraic set of $R\left(\pi_1(X_{2n})\right)$ with coordinates $\rho(s)$ and $\rho(t)$ we have the defining equation of 
$R\left(\pi_1(X_{2n})\right)$. By plugging in $e^{\frac{i \alpha}{2}}$ into $M$ of that equation, we have the following theorem.

\begin{theorem} \label{thm:RMpolynomial}
$\rho$ is a representation of $\pi_1(X_{2n})$ if and only if $x$ is a root of the following Riley-Mednykh polynomial $P_{2n}=P_{2n}(x,M)$ which is given recursively by 

\medskip
\begin{equation*}
P_{2n} = \begin{cases}
 Q P_{2(n-1)} -M^8 P_{2(n-2)} \ \text{if $n>1$}, \\
 Q P_{2(n+1)}-M^8 P_{2(n+2)} \ \text{if $n<-1$},
\end{cases}
\end{equation*}
\medskip

\noindent with initial conditions
\begin{equation*}
\begin{split}
P_{-2}  & =M^2 x^2+\left(M^4- M^2+1\right) x+M^2,\\
P_{0}  & =M^{-2} \ \text{\textnormal{for}} \ n < 0 \qquad \text{\textnormal{and}} \qquad P_{0} (x,M)  =1 \ \text{\textnormal{for}} \ n > 0,\\    
P_{2}  & =-M^4 x^3+\left(-2M^6+M^4-2 M^2\right) x^2+\left(-M^8+M^6-2 M^4+ M^2-1\right) x+M^4, \\
\end{split}
\end{equation*}
\noindent and
$$Q=-M^4 x^3+\left(-2M^6+2 M^4-2M^2\right) x^2+\left(-M^8+2M^6-3M^4+2M^2-1\right) x+2M^4.$$
\end{theorem}

\begin{proof}
Note that $SWT^{-1}W^{-1}=I$, which gives the defining equations of 
$R\left(\pi_1(X_{2n})\right)$, is equivalent to $(SWc)^2=-I$ in $\text{SL}(2,\C)$ 
by Lemma~\ref{lem:swc} and  
$(SWc)^2=-I$ in $\text{SL}(2,\C)$ is equivalent to $\text{\textnormal{tr}}(SWc)=0$.

We can find  two $c$'s in $\text{\textnormal{SL}}(2, \C)$ which satisfies $cS=T^{-1}c$ and $c^2=-I$ by direct computations. The existence and the uniqueness of the isometry (the involution) which is represented by $c$ are shown in~\cite[p. 46]{F}. Since two $c$'s give the same element in $\text{\textnormal{PSL}}(2, \C)$, we use one of them.
Hence, we may assume
 \begin{center}
$$\begin{array}{cc}
c=\left[\begin{array}{cc}
        0 & -\frac{M }{\sqrt{-1+2 M^2-M^4-M^2 x}}     \\
        \frac{\sqrt{-1+2 M^2-M^4-M^2 x}}{M} & 0
       \end{array}  \right],
\end{array}$$
\end{center}

 \begin{center}
$$\begin{array}{ccccc}
S=\left[\begin{array}{cc}
     e^ {\frac{ i \alpha}{2} }& 1       \\
     0 &  e^{-\frac{i  \alpha}{2}} 
                     \end{array} \right],                          
\ \ \
T=\left[\begin{array}{cc}
     e^ { \frac{i \alpha}{2} }& 0       \\
     2-e^{i \alpha}-e^{-i \alpha}-x &  e^{ -\frac{i \alpha}{2}} 
                     \end{array} \right],            
\end{array}$$
\end{center}

and let $U=TS^{-1}TST^{-1}S$. 

Recall that $P_{2n}$ is the defining polynomial of the algebraic set $\{(M,x)\}$ and the defining polynomial of $R\left(\pi_1(X_{2n})\right)$ corresponding to our choice of $\rho(s)$ and $\rho(t)$. We will show that 
$P_{2n}$ is equal to $\text{\textnormal{tr}}(SWc) / \text{\textnormal{tr}}(Sc)$ times $M^{4n}$ for $n>0$ and $M^{4n-2}$ for $n<0$. One can easily see 
$\text{\textnormal{tr}}(SUc)=P_2 \text{\textnormal{tr}}(Sc)/M^4$, $\text{\textnormal{tr}}(SWc)=P_{-2} \text{\textnormal{tr}}(Sc)/M^2$ and $\text{\textnormal{tr}}(U)=Q/M^4=\text{\textnormal{tr}}(U^{-1})$. We set $P_0$ as the statement of the theorem.
Now, we only need the following recurrence relations.

\begin{align*}
\text{\textnormal{tr}}(SWc) & =\text{\textnormal{tr}}(SU^nc)=\text{\textnormal{tr}}(SU^{n-1}cU^{-1})=\text{\textnormal{tr}}(SU^{n-1}c)\text{\textnormal{tr}}(U^{-1})-\text{\textnormal{tr}}(SU^{n-1}cU) \\
            & =\text{\textnormal{tr}}(SU^{n-1}c)\text{\textnormal{tr}}(U^{-1})-\text{\textnormal{tr}}(SU^{n-2}c)=0 \ \text{if $n>1$}
\end{align*}
or
\begin{equation*}
\begin{split}
\text{\textnormal{tr}}(SWc) & =\text{\textnormal{tr}}(SU^nc)=\text{\textnormal{tr}}(SU^{n+1}cU)=\text{\textnormal{tr}}(SU^{n+1}c)\text{\textnormal{tr}}(U)-\text{\textnormal{tr}}(SU^{n+1}cU^{-1}) \\
            & =\text{\textnormal{tr}}(SU^{n+1}c)\text{\textnormal{tr}}(U)-\text{\textnormal{tr}}(SU^{n+2}c)=0 \ \text{if $n<-1$},
\end{split}
\end{equation*}
where the third equality comes from the Cayley-Hamilton theorem. Since
$\text{\textnormal{tr}}(Sc)$, $\text{\textnormal{tr}}(SUc)$, and $\text{\textnormal{tr}}(SU^{-1}c)$ have 
$\text{\textnormal{tr}}(Sc)=\frac{\sqrt{-1+2 M^2-M^4-M^2 x}}{M}$
 as a common factor, all of 
$\text{\textnormal{tr}}(SWc)$'s have $\text{\textnormal{tr}}(Sc)$ as a common factor. We left the common factor out of $\text{\textnormal{tr}}(SWc)$, multiplied it by a power of $M^{4n}$ if $n>0$ and $M^{4n-2}$ for $n<0$ to clear the fractions and denote it by $P_{2n}$, the Riley-Mednykh polynomial.
\end{proof}

\section{Longitude}
\label{sec:longitude}
   Let $l = ww^{*}s^{-4n}$, where $w^{*}$ is the word obtained by reversing $w$. Let $L=\rho(l)_{11}$. Then $l$ is the longitude which is null-homologus in $X_{2n}$. Recall $\rho(w)=U^n$. We can write $\rho(w^{*})=\widetilde{U}^n$. It is easy to see that $U$ and $\widetilde{U}$ can be written as
   
   $$U=\begin{pmatrix}
u_{11} & u_{12} \\
u_{21} & u_{22}
\end{pmatrix}$$
and 
$$\widetilde{U}=
\begin{pmatrix}
\tilde{u}_{22} & \tilde{u}_{12} \\
\tilde{u}_{21}& \tilde{u}_{11}
\end{pmatrix}
$$
   
where $\tilde{u}_{ij}$ is obtained by $u_{ij}$ by replacing $M$ with $M^{-1}$. Similar computation was introduced in~\cite{HS}.

\begin{definition} \label{def:longitude}
 The \emph{complex length} of the longitude $l$ is the complex number 
 $\gamma_{\alpha}$ modulo $4 \pi \Z$ satisfying 
\begin{align*}
 \text{\textnormal{tr}}(\rho(l))=2 \cosh \frac{\gamma_{\alpha}}{2}.
\end{align*}
 Note that 
 $l_{\alpha}=|Re(\gamma_{\alpha})|$ is the real length of the longitude of the cone-manifold $X_{2n}(\alpha)$.
\end{definition}

The following lemma was introduced in~\cite{HS} with slightly different coordinates. 

\begin{lemma}~\cite{HS}
\label{lem:lemma}
$u_{21} L+\tilde{u}_{21} M^{-4n}=0$.
\end{lemma}

\begin{theorem}
\label{thm:longitude}
\begin{align*}
L=-M^{-4n-2}\frac{M^{-2}+x}{ M^{2}+x}.
\end{align*}
\end{theorem}

\begin{proof}
By directly computing $u_{21} L+\tilde{u}_{21} M^{-4n}=0$ in Lemma~\ref{lem:lemma}, the theorem follows.
\end{proof}


\section{Proof of Theorem~\ref{thm:main}} \label{sec:proof}

For $n \geq 1$ and $M=e^{i \frac{\alpha}{2}}$, $P_{2n}(x,M)$ have $3n$ component zeros, and for  $n < -1$, $-(3n+1)$ component zeros. For each $n$, there exists an angle $\alpha_0 \in [\frac{2\pi}{3},\pi)$ such that $T_{2n}(\alpha)$ is hyperbolic for $\alpha \in (0, \alpha_0)$, Euclidean for $\alpha=\alpha_0$, and spherical for $\alpha \in (\alpha_0, \pi]$ \cite{P2,HLM1,K1,PW}. 
From the following Equality~(\ref{equ:absL}), when $|L|=1$, which happens when $\alpha=\alpha_0$, $\text{\textnormal{Im}}(x)=0$. Hence, when 
$\alpha$ increases from $0$ to $\alpha_0$, two complex numbers $x$ and $\overline{x}$ approach a same real number. In other words, 
$P_{2n}(x,e^{\frac{i \alpha_0}{2}})$ has a multiple root.
Denote by $D(X_{2n}(\alpha))$  the discriminant of
$P_{2n}(x,M)$ over $x$. Then $\alpha_0$ will be one of the zeros of $D(X_{2n}(\alpha))$.

From Theorem~\ref{thm:longitude}, we have the following equality,
\begin{equation}\label{equ:absL}
\begin{split}
|L|^2 &= \left|\frac{M^{-2}+x}{ M^2+x}\right| 
= \frac{(\cos{\alpha}+{\rm Re}(x))^2 + ({\rm Im}(x)-\sin{\alpha})^2}{(\cos{\alpha}+{\rm Re}(x))^2 + ({\rm Im}(x)+ \sin{\alpha})^2}.
\end{split}
\end{equation}

For the volume, we choose $L$ with $|L|\geq1$ and hence we have $\text{\textnormal{Im}}(x) \leq 0$ by Equality~(\ref{equ:absL}). 
 Using the Schl\"{a}fli formula, we calculate the volume of 
$\rho(X_{2n})=\rho(X_{2n}(0))$ for each component with $|L|\geq1$ and having one of the zeros $\alpha_0$ of $D(X_{2n}(\alpha))$ with $\alpha_0 \in [\frac{2\pi}{3},\pi)$ on it. The component which gives the maximal volume is the excellent component~\cite{D1,FK1}. On the geometric component we have
 the volume of a hyperbolic cone-manifold 
$X_{2n}(\alpha)$ for $0 \leq \alpha < \alpha_0$:
\begin{align*}
\text{\textrm{Vol}}(X_{2n}(\alpha)) &=-\int_{\alpha_0}^{\alpha} \frac{l_{\alpha}}{2} \: d\alpha \\
                        &=-\int_{\alpha_0}^{\alpha} \log|L| \: d\alpha\\
                         &=-\int_{\pi}^{\alpha} \log|L| \: d\alpha\\
                         &=\int^{\pi}_{\alpha} \log|L| \: d\alpha\\
                         &=\int^{\pi}_{\alpha}  \log \left|\frac{M^{-2}+x}{ M^2+x}\right|\: d\alpha,                       
\end{align*}
where the first equality comes from the Schl\"{a}fli formula for cone-manifolds (Theorem 3.20 of~\cite{CHK}), the second equality comes from the fact that $l_{\alpha}=|Re(\gamma_{\alpha})|$ is the real length of the longitude of 
$X_{2n}(\alpha)$, the third equality comes from the fact that $\log|L|=0$  for $\alpha_0 < \alpha \leq \pi$ by Equality~\ref{equ:absL} since all the characters are real (the proof of Proposition 6.4 of~\cite{PW}) for $\alpha_0 < \alpha \leq \pi$, and 
$\alpha_0 \in [\frac{2 \pi}{3},\pi)$ is a zero of the discriminant $D(X_{2n}(\alpha))$.

\medskip


\section*{Acknowledgments}
The authors would like to thank Alexander Mednykh and Hyuk Kim for their various helps and anonymous referees.


\begin{thebibliography}{99}

\bibitem{CHK}
Daryl Cooper, Craig~D. Hodgson, and Steven~P. Kerckhoff.
\newblock {\em Three-dimensional orbifolds and cone-manifolds}, volume~5 of
  {\em MSJ Memoirs}.
\newblock Mathematical Society of Japan, Tokyo, 2000.
\newblock With a postface by Sadayoshi Kojima.

\bibitem{DMM1}
D.~Derevnin, A.~Mednykh, and M.~Mulazzani.
\newblock Volumes for twist link cone-manifolds.
\newblock {\em Bol. Soc. Mat. Mexicana (3)}, 10(Special Issue):129--145, 2004.

\bibitem{D1}
Nathan~M. Dunfield.
\newblock Cyclic surgery, degrees of maps of character curves, and volume
  rigidity for hyperbolic manifolds.
\newblock {\em Invent. Math.}, 136(3):623--657, 1999.

\bibitem{F}
Werner Fenchel.
\newblock {\em Elementary geometry in hyperbolic space}, volume~11 of {\em de
  Gruyter Studies in Mathematics}.
\newblock Walter de Gruyter \& Co., Berlin, 1989.
\newblock With an editorial by Heinz Bauer.

\bibitem{FK1}
Stefano Francaviglia and Ben Klaff.
\newblock Maximal volume representations are {F}uchsian.
\newblock {\em Geom. Dedicata}, 117:111--124, 2006.

\bibitem{HL}
Ji-Young Ham and Joongul Lee.
\newblock Explicit formulae for {C}hern-{S}imons invariants of the twist knot
  orbifolds and edge polynomials of twist knots.
\newblock \url{http://arxiv.org/abs/1411.2383}, 2014.
\newblock Preprint.

\bibitem{HMP}
Ji-Young Ham, Alexander Mednykh, and Vladimir Petrov.
\newblock Trigonometric identities and volumes of the hyperbolic twist knot
  cone-manifolds.
\newblock {\em J. Knot Theory Ramifications}, 23(12):1450064, 16, 2014.

\bibitem{HLM1}
Hugh Hilden, Mar{\'{\i}}a~Teresa Lozano, and Jos{\'e}~Mar{\'{\i}}a
  Montesinos-Amilibia.
\newblock On a remarkable polyhedron geometrizing the figure eight knot cone
  manifolds.
\newblock {\em J. Math. Sci. Univ. Tokyo}, 2(3):501--561, 1995.

\bibitem{HLM2}
Hugh~M. Hilden, Mar{\'{\i}}a~Teresa Lozano, and Jos{\'e}~Mar{\'{\i}}a
  Montesinos-Amilibia.
\newblock Volumes and {C}hern-{S}imons invariants of cyclic coverings over
  rational knots.
\newblock In {\em Topology and {T}eichm\"uller spaces ({K}atinkulta, 1995)},
  pages 31--55. World Sci. Publ., River Edge, NJ, 1996.

\bibitem{HS}
Jim Hoste and Patrick~D. Shanahan.
\newblock A formula for the {A}-polynomial of twist knots.
\newblock {\em J. Knot Theory Ramifications}, 13(2):193--209, 2004.

\bibitem{K1}
Sadayoshi Kojima.
\newblock Deformations of hyperbolic {$3$}-cone-manifolds.
\newblock {\em J. Differential Geom.}, 49(3):469--516, 1998.

\bibitem{K2}
Sadayoshi Kojima.
\newblock Hyperbolic {$3$}-manifolds singular along knots.
\newblock {\em Chaos Solitons Fractals}, 9(4-5):765--777, 1998.
\newblock Knot theory and its applications.

\bibitem{M2}
Alexander Mednykh.
\newblock The volumes of cone-manifolds and polyhedra.
\newblock \url{http://mathlab.snu.ac.kr/~top/workshop01.pdf}, 2007.
\newblock {L}ecture {N}otes, Seoul National University.

\bibitem{MR2}
Alexander Mednykh and Aleksey Rasskazov.
\newblock On the structure of the canonical fundamental set for the 2-bridge
  link orbifolds.
\newblock \url{www.mathematik.uni-bielefeld.de/sfb343/preprints/pr98062.ps.gz},
  1998.
\newblock Universität Bielefeld, Sonderforschungsbereich 343, Discrete
  Structuren in der Mathematik, Preprint, 98–062.

\bibitem{MR1}
Alexander Mednykh and Alexey Rasskazov.
\newblock Volumes and degeneration of cone-structures on the figure-eight knot.
\newblock {\em Tokyo J. Math.}, 29(2):445--464, 2006.

\bibitem{MV1}
Alexander Mednykh and Andrei Vesnin.
\newblock On the volume of hyperbolic {W}hitehead link cone-manifolds.
\newblock {\em Sci. Ser. A Math. Sci. (N.S.)}, 8:1--11, 2002.
\newblock Geometry and analysis.

\bibitem{M1}
Alexander~D. Mednykh.
\newblock Trigonometric identities and geometrical inequalities for links and
  knots.
\newblock In {\em Proceedings of the {T}hird {A}sian {M}athematical
  {C}onference, 2000 ({D}iliman)}, pages 352--368. World Sci. Publ., River
  Edge, NJ, 2002.

\bibitem{P2}
Joan Porti.
\newblock Spherical cone structures on 2-bridge knots and links.
\newblock {\em Kobe J. Math.}, 21(1-2):61--70, 2004.

\bibitem{PW}
Joan Porti and Hartmut Weiss.
\newblock Deforming {E}uclidean cone 3-manifolds.
\newblock {\em Geom. Topol.}, 11:1507--1538, 2007.

\bibitem{R1}
Robert Riley.
\newblock Parabolic representations of knot groups. {I}.
\newblock {\em Proc. London Math. Soc. (3)}, 24:217--242, 1972.

\bibitem{R3}
Robert Riley.
\newblock Nonabelian representations of {$2$}-bridge knot groups.
\newblock {\em Quart. J. Math. Oxford Ser. (2)}, 35(138):191--208, 1984.

\bibitem{S}
Horst Schubert.
\newblock Knoten mit zwei {B}r\"ucken.
\newblock {\em Math. Z.}, 65:133--170, 1956.

\end{thebibliography}
\end{document}